\documentclass[10pt,reqno]{amsart}
\usepackage{amssymb}
\usepackage[all]{xy}

\theoremstyle{plain}
\newtheorem{prop}{Proposition}
\newtheorem{theo}[prop]{Theorem}

\newtheorem{lemm}[prop]{Lemma}
\theoremstyle{remark}

\theoremstyle{definition}

\numberwithin{equation}{section}

\newcommand{\A}{{\mathbb A}}

\newcommand{\PP}{{\mathbb P}}
\newcommand{\Q}{{\mathbb Q}}
\newcommand{\G}{{\mathbb G}}

\newcommand{\Z}{{\mathbb Z}}

\newcommand{\eqto}{\stackrel{\lower1.5pt\hbox{$\scriptstyle\sim\,$}}\to}
\DeclareMathOperator{\Gal}{Gal}
\DeclareMathOperator{\inv}{inv}

\DeclareMathOperator{\Pic}{Pic}

\DeclareMathOperator{\Br}{Br}

\DeclareMathOperator{\chara}{char}

\begin{document}
\title[Brauer--Manin obstructions on surfaces]
{Effectivity of Brauer--Manin obstructions on surfaces}
\author{Andrew Kresch}
\address{
  Institut f\"ur Mathematik,
  Universit\"at Z\"urich,
  Winterthurerstrasse 190,
  CH-8057 Z\"urich, Switzerland
}
\email{andrew.kresch@math.uzh.ch}
\author{Yuri Tschinkel}
\address{
  Courant Institute,
  251 Mercer Street,
  New York, NY 10012, USA
}
\email{tschinkel@cims.nyu.edu}

\date{May 24, 2010}
\begin{abstract}
We study Brauer--Manin obstructions to the Hasse principle and to
weak approximation on algebraic surfaces over number fields.
\end{abstract}
\maketitle

\section{Introduction}
\label{sec:introduction}

Let $X$ be a smooth projective variety over a number field $k$.
An important area of research concerns the behavior of the set of $k$-rational
points $X(k)$.
One of the major open problems is the decidability problem for
$X(k)\ne\emptyset$.
An obvious necessary condition is the existence of points over all
completions $k_v$ of $k$; this can be effectively tested given
defining equations of $X$.
One says that $X$ \emph{satisfies the Hasse principle} when
\begin{equation}
\label{eqn:hasseprinciple}
X(k)\ne \emptyset \Leftrightarrow X(k_v)\ne \emptyset \,\,\,\forall v.
\end{equation}

One well-studied obstruction to this is the
\emph{Brauer--Manin obstruction} \cite{manin}.
It has proved remarkably useful in explaining counterexamples to the
Hasse principle, especially on curves \cite{stoll}
and geometrically rational surfaces \cite{CTSSD};
see also \cite{skorobook}.
Although there are counterexamples not explained by the Brauer--Manin obstruction
\cite{skoro}, \cite{poonen}, there remains a wide class of algebraic varieties
for which the sufficiency of the Brauer--Manin obstruction is a subject
of active research.
This includes $K3$ surfaces, studied for instance in
\cite{SD00},
\cite{wittenberg},
\cite{harariskorobogatov},
\cite{skorobogatovswinnertondyer},
\cite{Bri06},
\cite{ieronymou},
\cite{HVV}.

We recall, that an element $\alpha\in \Br(X)$ cuts out a subspace
$$X(\A_k)^\alpha\subseteq X(\A_k)$$
of the adelic space, defined as the set of all $(x_v)\in X(\A_k)$ satisfying
$$\sum_v \inv_v(\alpha(x_v))=0.$$
Here, $\inv_v$ is the local invariant of the restriction of $\alpha$ to a
$k_v$-point, taking its value in $\Q/\Z$.
By the exact sequence of class field theory
$$0\to \Br(k)\to \bigoplus_v \Br(k_v)\stackrel{\inv}\to \Q/\Z\to 0$$
(here $\inv$ is the sum of $\inv_v$),
we have
$$X(k)\subseteq X(\A_k)^\alpha.$$
Therefore, for any subset $\mathrm{B}\subseteq \Br(X)$ we have
$$X(k)\subseteq X(\A_k)^{\mathrm{B}}:=\bigcap_{\alpha\in \mathrm{B}}X(\A_k)^\alpha.$$

A natural goal is to be able to compute the space $X(\A_k)^{\Br(X)}$
\emph{effectively}.
By this we mean, to give an algorithm, for which there is an \emph{a priori}
bound on the running time, in terms of the input data (e.g., the defining equations
of $X$).
The existence of such an effective algorithm
was proved for geometrically rational surfaces in \cite{KTeff}.
Here we prove the following result.

\begin{theo}
\label{thm.main}
Let $X$ be a smooth projective geometrically irreducible surface over a number field $k$,
given by a system of homogeneous polynomial equations.
Assume that
the geometric Picard group $\Pic(X_{\bar k})$ is torsion free and
generated by finitely many divisors, each with a given set of
defining equations.
Then for each positive integer $n$ there exists an effective
description of a space $X_n\subseteq X(\A_k)$ which satisfies
$$
X(\A_k)^{\Br(X)}\subseteq X_n \subseteq X(\A_k)^{\Br(X)[n]},
$$
where $\Br(X)[n]\subseteq\Br(X)$ denotes the $n$-torsion subgroup.
In particular, $X(\A_k)^{\Br(X)}$ is effectively computable provided that
$|\Br(X)/\Br(k)|$ can be bounded effectively.
\end{theo}

For instance, in the case of a diagonal quartic surface over $\Q$ there
is an effective bound on $|\Br(X)/\Br(\Q)|$ due to
Ieronymou, Skorobogatov, and Zarhin \cite{ISZ}.

While it is not known how to compute $\Pic(X_{\bar k})$ effectively,
in general, there is a method of computation involving reduction
modulo primes used by van Luijk \cite{vL}; further examples can be found
in \cite{EJ} and \cite{HVV}.

\medskip
\noindent\textbf{Acknowledgements.}
The first author was supported by the SNF.
The second author was supported by NSF grants 0739380 and 0901777.
Both authors acknowledge the gracious support of the
Forschungsinstitut f\"ur Mathematik of the ETH, which hosted the second
author for a research visit in Zurich.

\section{Picard schemes}
\label{sec.picardschemes}
Let $X\to S$ be a finite-type morphism of locally Noetherian schemes.
We recall that the functor associating to an $S$-scheme $T$ the group
$$\Pic_{X/S}(T):=\Pic(X\times_ST)/\Pic(T)$$
is known as the \emph{relative Picard functor}.
It restricts to a sheaf on the \'etale site $S_{\mathrm{et}}$ when $S$ is a nonsingular curve over
an algebraically closed field, by Tsen's theorem.
See \cite{kleiman}.

We use $\Br(X)$ to denote the cohomological Brauer group of a Noetherian scheme $X$,
i.e., the torsion subgroup of the \'etale cohomology group $H^2(X,\G_m)$.
When $X$ is regular, $H^2(X,\G_m)$ is itself a torsion group.
By Gabber's theorem, if $X$ admits an ample invertible sheaf then
$\Br(X)$ is also equal to the Azumaya Brauer group, i.e., the equivalence
classes of sheaves of Azumaya algebras on $X$.
For background on the Brauer group, the reader is referred to
\cite{gb}, and for a proof of Gabber's theorem, see \cite{dejong}.

Let $S$ be a nonsingular irreducible curve over an algebraically closed field,
and let $f:X\to S$ be a smooth projective morphism of relative dimension
1 with connected fibers.
Then the Leray spectral sequence
$$E_2^{p,q}=H^p(S,R^qf_*\G_m) \Longrightarrow H^{p+q}(X,\G_m)$$
gives, by \cite[Cor.\ III.3.2]{gb}, an isomorphism
\begin{equation}
\label{eqn.lerayexaseq}
\Br(X)\eqto H^1(S,\Pic_{X/S}).
\end{equation}
Furthermore, we have an exact sequence
$$0\to \Pic^0_{X/S}\to \Pic_{X/S}\to \Z\to 0$$
of sheaves (on $S_{\mathrm{et}}$) hence an exact sequence
\begin{equation}
\label{eqn.withd}
0\to \Z/d\Z\to H^1(S,\Pic^0_{X/S})\to H^1(S,\Pic_{X/S})\to 0,
\end{equation}
where $d$ is the gcd of the relative degrees of all multisections of $f$.
Now assume that the algebraically closed base field has characteristic
not dividing $n$.
Then we have the exact sequence of sheaves
$$0\to \Pic_{X/S}[n]\to \Pic^0_{X/S}\stackrel{n\cdot}\to \Pic^0_{X/S}\to 0$$
(exactness on the right follows by
\cite[21.9.12]{EGAIV} and \cite[Prop.\ 9.5.19]{kleiman})
from which the long exact sequence in cohomology gives a surjective homomorphism
\begin{equation}
\label{eqn.ntorsion}
H^1(S,\Pic_{X/S}[n])\to H^1(S,\Pic^0_{X/S})[n].
\end{equation}

\begin{lemm}
\label{lem.H1Cmun}
Let $K$ be a field, and let $D$ be a geometrically
irreducible smooth projective curve over $K$.
Let $n$ be a positive integer, not divisible by $\chara(K)$.
Let $C$ be a nonempty open subset of $D$, with $Y:=D\smallsetminus C$
nonempty.
The inclusions will be denoted $i\colon Y\to D$ and $j\colon C\to D$.
\begin{itemize}
\item[(i)] We have $R^1j_*\mu_n=i_*(\Z/n\Z)$.
\item[(ii)] For a tuple of integers $(a_y)_{y\in Y}$ with
reductions $(\bar a_y)$ modulo $n$, we have $(\bar a_y)$ in the image of
the map
$$H^1(C,\mu_n)\to H^0(D,R^1j_*\mu_n) = \bigoplus_{y\in Y}\Z/n\Z$$
coming from the Leray spectral sequence if and only if
there exists a divisor $\delta$ on $D$ with $n\delta\sim \sum a_y[y]$,
where $\sim$ denotes linear equivalence of divisors.
\end{itemize}
\end{lemm}

\begin{proof}
The Leray spectral sequence gives
\begin{equation}
\label{eqn.lerayCD}
0\to H^1(D,\mu_n)\to H^1(C,\mu_n)\to H^0(D,R^1j_*\mu_n)
\stackrel{d_2^{0,1}}\to H^2(D,\mu_n)\to H^2(C,\mu_n).
\end{equation}
For (i), by standard spectral sequences we have
$R^1j_*\mu_n=i_*\underline{H}^2_Y(\mu_n)$
(cf.\ \cite[proof of Thm.\ VI.5.1]{milne}).
So we are reduced to a local computation, and we may therefore assume that
$D$ is affine and $Y$ consists of a single point which is a principal
Cartier divisor on $D$.
By the Kummer sequence and injectivity of $\Br(C)\to \Br(D)$
the right-hand map in \eqref{eqn.lerayCD} is injective, while
the left-hand map has cokernel cyclic of order $n$.
(Such an isomorphism exists generally for regular codimension $1$
complements, see \cite[(XIX.3.3)]{sga4}.)

For the ``if'' direction of (ii), we take $r\in K(D)^*$ to be
a rational function whose divisor is $-n\delta + \sum a_y[y]$.
Then adjoining $r^{1/n}$ to the function field of $D$ yields an
element of $H^1(C,\mu_n)$ whose image in $H^0(D,R^1j_*\mu_n)$ is
$(\bar a_y)$ by the isomorphism in (i).
For the ``only if'' direction, an element of $H^1(C,\mu_n)$ gives rise
by the Kummer exact sequence to a divisor $\delta$ on $C$ and
$r\in K(C)^*$ by which $n\delta\sim 0$ on $C$.
Then $n\delta\sim \sum b_y[y]$ on $D$, for some integers $b_y$,
and the given element of $H^1(C,\mu_n)$ maps by $d^{0,1}_2$ to
$(\bar b_y)$.
This means that $a_y\equiv b_y$ mod $n$ for all $y\in Y$, and we
easily obtain $\delta'$ on $D$ with $n\delta'\sim \sum a_y[y]$.
\end{proof}

\section{Brauer groups}
\label{sec.brauer}

We start with some general results about cocycles in \'etale cohomology.

\begin{lemm}
\label{lem.effmv}
Let $X$ be a Noetherian scheme,
union of open subschemes $X_1$ and $X_2$,
and let $G$ be an abelian \'etale sheaf.
Suppose given \'etale covers $Y_i\to X_i$ and \v{C}ech cocycles
$\beta_i\in Z^2(Y_i\to X_i,G)$ for $i=1$, $2$.
With $X_{12}=X_1\cap X_2$ and $Y_{12}=Y_1\times_X Y_2$, we suppose
further that a cochain $\delta\in C^1(Y_{12}\to X_{12},G)$ is given, satisfying
\[
\frac{\delta(y_1,y'_1,y_2,y'_2)\delta(y'_1,y''_1,y'_2,y''_2)}{\delta(y_1,y''_1,y_2,y''_2)}
=\frac{\beta_1(y_1,y'_1,y''_1)}{\beta_2(y_2,y'_2,y''_2)}
\]
for $(y_1,y_2,y'_1,y'_2,y''_1,y''_2)\in Y_{12}\times_XY_{12}\times_XY_{12}$.
Then we have $\beta\in Z^2(Y_1 \amalg Y_2\to X,G)$, given by
\begin{align*}
(y_1,y'_1,y''_1) &\mapsto \beta_1(y_1,y'_1,y''_1), \\
(y_1,y'_1,y''_2) &\mapsto \delta(y_1,y'_1,y''_2,y''_2)\beta_2(y''_2,y''_2,y''_2), \\
(y_1,y'_2,y''_1) &\mapsto \delta(y_1,y''_1,y'_2,y'_2)^{-1}\beta_2(y'_2,y'_2,y'_2)^{-1}, \\
(y_1,y'_2,y''_2) &\mapsto \delta(y_1,y_1,y'_2,y''_2)^{-1}\beta_1(y_1,y_1,y_1), \\
(y_2,y'_1,y''_1) &\mapsto \delta(y'_1,y''_1,y_2,y_2)\beta_2(y_2,y_2,y_2), \\
(y_2,y'_1,y''_2) &\mapsto \delta(y'_1,y'_1,y_2,y''_2)\beta_1(y'_1,y'_1,y'_1)^{-1}, \\
(y_2,y'_2,y''_1) &\mapsto \delta(y''_1,y''_1,y_2,y'_2)^{-1}\beta_1(y''_1,y''_1,y''_1), \\
(y_2,y'_2,y''_2) &\mapsto \beta_2(y_2,y'_2,y''_2),
\end{align*}
whose class restricts to the class of $\beta_i$ in $H^2(X_i,G)$ for
$i=1$, $2$.
\end{lemm}

\begin{proof}
This is just a portion of the Mayer-Vietoris sequence, written out explicitly
in terms of cocycles.
\end{proof}

The following two results are based on the existence of Zariski local trivializations
of $1$-cocycle with values in $\G_m$.
Such trivializations exist effectively when
the $1$-cocycle is effectively presented,
say on a scheme of finite type over a number field.

\begin{lemm}
\label{lem.cobounda}
Let $k$ be a number field, $X$ a finite-type scheme over $k$,
$Y\to X$ and $Z\to X$ finite-type \'etale covers,
and $Y\to Z$ a morphism over $X$.
Suppose that
$\beta\in Z^2(Z\to X,\G_m)$ and $\delta\in C^1(Y\to X,\G_m)$ are given, so
that the restriction of $\beta$ by $Y\to Z$ is equal to the
coboundary of $\delta$.
Then we may effectively produce a Zariski open covering
$Z=\bigcup_{i=1}^N Z_i$ for some $N$ and a
$1$-cochain for $\coprod_{i=1}^N Z_i\to X$ whose coboundary is
equal to the restriction of $\beta$ by $\coprod_{i=1}^N Z_i\to Z$.
\end{lemm}

\begin{proof}
Replacing $Y$ by $Y\times_XZ$ and using that the restriction maps on the
level of \v{C}ech cocycles corresponding to
$Y\times_XZ\to Y\to Z$ and $Y\times_XZ\to Z$ differ by an explicit
coboundary (cf. \cite[Lem.\ III.2.1]{milne}), we are reduced to the
case that $Y\to Z$ is also a covering.

Then we have the $1$-cocycle for $Y\to Z$
$$(y,y')\mapsto \frac{\beta(z,z,z)}{\delta(y,y')},$$
for $(y,y')\in Y\times_ZY$ over $z$.
This may be trivialized effectively on a Zariski open neighborhood of
any point of $Z$, so we may effectively obtain a refinement of $Z$ to a
Zariski open covering and functions $\varepsilon_i$ satisfying
$$\frac{\varepsilon_i(y')}{\varepsilon_i(y)}=
\frac{\beta(z,z,z)}{\delta(y,y')}$$
for all $i$ and $(y,y')\in Y\times_ZY$ over $z\in Z_i$.
It follows that for all $i$ and $j$, and
$(y,y')\in Y\times_XY$ over $(z,z')\in Z\times_XZ$ with
$z\in Z_i$, $z'\in Z_j$, the function
$$\frac{\varepsilon_j(y')}{\varepsilon_i(y)}\delta(y,y')$$
depends only on $(z,z')$, hence we obtain
$\delta_0\in C^1(\coprod_{i=1}^N Z_i\to X)$
satisfying
$$\frac{\varepsilon_j(y')}{\varepsilon_i(y)}\delta(y,y')=\delta_0(z,z').$$
The conclusion follows immediately from this formula.
\end{proof}

\begin{lemm}
\label{lem.effBring}
Let $X$ be a smooth finite-type scheme over a number field $k$,
let $Z\to X$ be a finite-type \'etale covering, and let $Y\to Z$ be
a finite-type \'etale morphism with dense image.
Let $\beta\in Z^2(Z\to X,\G_m)$ be given, along with
$\delta\in\mathcal{O}_{Y\times_XY}^*$ satisfying
$$
\delta(y,y')\delta(y',y'')/\delta(y,y'')=\beta(z,z',z''),
$$
for all $(y,y',y'')\in Y\times_XY\times_XY$ over
$(z,z',z'')\in Z\times_XZ\times_XZ$.
Then there exists, effectively, a Zariski open covering
$(Z_i)_{1\le i\le N}$ of $Z$ (for some $N$) and a $1$-cocycle for
$\coprod Z_i\to X$ whose coboundary is the restriction of $\beta$
by $\coprod Z_i\to Z$.
\end{lemm}

\begin{proof}
Let $X_0$ denote the image of the composite morphism $Y\to X$,
and $Z_0$ the pre-image of $X_0$ in $Z$.
By Lemma \ref{lem.cobounda} (or rather its proof) there exists
a Zariski open covering of $Z_0$ of the form
$(Z_0\cap Z_i)_{1\le i\le N}$ for some Zariski open covering
$(Z_i)$ of $Z$ (the $1$-cocycle mentioned in the proof determines
a line bundle on $Z_0$, which can be extended to a line bundle
on $Z$, since $Z$ is smooth) and a $1$-cochain for
$\coprod Z_i\cap Z_0\to X_0$ whose coboundary is the restriction
of $\beta$.
Using the fact that divisors on smooth schemes are locally
principal (and effectively so, e.g., see \cite[\S 7]{KTeff})
and \cite[Exa.\ III.2.22]{milne},
we see that after further refinement of $(Z_i)$ the
$1$-cochain extends to a $1$-cochain for
$\coprod Z_i\to X$.
\end{proof}

Let $X$ be a regular Noetherian scheme of dimension $2$.
It is known \cite[Cor.\ II.2.2]{gb}
that for any element
$\alpha\in \Br(X)$ of the (cohomological) Brauer group
there exists a sheaf of Azumaya algebras on $X$
having class equal to $\alpha$.

\begin{lemm}
\label{lem.effAz}
Let $X$ be a smooth projective surface over a number field $k$,
$\widehat{X}\subset X$ an open subscheme whose complement has codimension 2,
and $\alpha\in \Br(X)$ an element whose restriction over $\widehat{X}$
is represented by a
$2$-cocycle $\hat\beta$, relative to some finite-type \'etale cover
$\pi\colon \widehat{Y}\to \widehat{X}$.
We suppose that
$X$, $\widehat{X}$, $\widehat{Y}$, $\pi$, and $\hat\beta$ are given by explicit equations.
Then there is an effective procedure to produce a sheaf of Azumaya algebras
on $X$ representing the class $\alpha$.
\end{lemm}

Note, by purity for the Brauer group \cite[Thm.\ III.6.1]{gb}, we have
$\Br(\widehat{X})=\Br(X)$, so $\alpha$ is uniquely determined by
the cocycle $\hat\beta$.

\begin{proof}
Take $V\subset \widehat{Y}$ nonempty open such that
$\psi_0=\pi|_V$ is a finite \'etale covering of some open subscheme
of $\widehat{X}$.
Let $\psi\colon \widehat{W}\to \widehat{X}$ be the
normalization of $\widehat{X}$ in $(\psi_0)_*\mathcal{O}_V$.
Shrinking $\widehat{X}$ (and maintaining that its complement
in $X$ has codimension 2) we may suppose that $\widehat{W}$ is smooth.
By the universal property of the normalization (\cite[6.3.9]{EGAII}) there
is a (unique) lift $\widehat{Y}\to \widehat{W}$ of $\pi$.
Consider the element of $Z^2(\widehat{Y}\times_{\widehat{X}}\widehat{W}\to \widehat{W},\G_m)$
obtained by restricting $\hat\beta$.
The further restriction to
$Z^2(\widehat{Y}\times_{\widehat{X}}\widehat{Y}\to \widehat{Y},\G_m)$ is
(explicitly) a coboundary,
we apply Lemma \ref{lem.effBring} and observe that by the
proof, from the fact that $\widehat{W}\to \widehat{X}$ is finite and
hence universally closed, the Zariski refinement may be
taken to come from a Zariski refinement of $\widehat{Y}$, i.e., we
obtain
$\hat\gamma\in C^1(\coprod \widehat{Y}_i\times_{\widehat{X}}\widehat{W}\to \widehat{W},\G_m)$
whose coboundary is the restriction of $\hat\beta$.
Using the flatness of $\widehat{W}\to \widehat{X}$, we may regard
$\hat\gamma$ as patching data for a
sheaf of Azumaya algebras over $\widehat{X}$
as in \cite[Prop.\ IV.2.11]{milne},
whose class in the Brauer group is that of $\hat\beta$.
Pushforward via $\widehat{X}\to X$ may be computed by making an
arbitrary extension as a coherent sheaf, and forming the double dual.
This is then a sheaf of Azumaya algebras on $X$ by \cite[Thm.\ I.5.1(ii)]{gb}.
\end{proof}

\section{Proof of Theorem \ref{thm.main}}
The proof of Theorem \ref{thm.main} is carried out in several steps.

\

\emph{Step 1.} (Proposition \ref{Brbyfibr}) We obtain a nonempty
open subscheme $X^\circ$ of $X$, a finite Galois extension
$K$ of $k$, and a sequence of elements
$$
(\alpha_1, \ldots, \alpha_N)\subset \Br(X^\circ_K)
$$ 
for some $N$
which generate a subgroup of $\Br(X^\circ_{\bar k})$ containing
$\Br(X_{\bar k})[n]$.
We obtain an \'etale covering $Y^\circ\to X^\circ$, such that
each $\alpha_i$ is given by an explicit 2-cocycle for the \'etale cover
$Y^\circ_K\to X^\circ_K$.

\

\emph{Step 2.} (Proposition \ref{prop.relation}) Given
$\alpha\in\Br(X^\circ_K)$ defined by an explicit cocycle,
we provide an effective procedure to test whether $\alpha$
vanishes in $\Br(X_{\bar k})$,
and in case of vanishing, to produce a $1$-cochain lift of the
cocycle, defined over some effective extension of $K$.
We use this procedure in two ways.
\begin{itemize}
\item[(i)] By repeating Step 1 with another open subscheme
$\widetilde{X}^\circ$, with $X\smallsetminus(X^\circ\cup \widetilde{X}^\circ)$
of codimension 2 (or empty), to identify the geometrically \emph{unramified}
Brauer group elements, i.e., those in the
image of $\Br(X_K)\to \Br(X^\circ_K)$ after possibly extending $K$.
\item[(ii)] To identify those $\alpha$ such that
$\alpha$ and ${}^g\alpha$ have the same image in $\Br(X_{\bar k})$ for all
$g\in \Gal(K/k)$.
Again after possibly extending $K$ (remaining finite Galois over $k$),
we may suppose that all such $\alpha$ satisfy
$\alpha={}^g\alpha$ in $\Br(X_K)$ for all $g\in \Gal(K/k)$.
\end{itemize}
The result is a sequence of elements
$$
(\alpha'_1,\ldots,\alpha'_M) \subset \Br(X_K)[n]^{\Gal(K/k)},
$$
each given by
a cocycle over $X^\circ_K$ as well as one over $\widetilde{X}^\circ_K$,
generating
$\Br(X_{\bar k})[n]^{\Gal(\bar k/k)}$.

\

\emph{Step 3.}
(Proposition \ref{prop.extend})
Combine the data from the Galois invariance of the $\alpha'_i$
and the alternate representation over $\widetilde{X}^\circ_K$ to obtain
cocycle representatives of each $\alpha'_i$ defined over the
complement of a codimension 2 subset of $X$, as well as cochains
there that encode the Galois invariance.

\

\emph{Step 4.} (Proposition \ref{prop.descent})
For every Galois-invariant $n$-torsion element of $\Br(X_{\bar{k}})$,
with representing cocycle defined over $K$ obtained in Step 3,
compute the obstruction to the existence of an element of
$\Br(X)$ having the same image class in $\Br(X_{\bar{k}})$.
When the obstruction vanishes, produce a cocycle representative
of such an element of $\Br(X)$, defined over the complement of
a codimension 2 subset of $X$.
Each such element of $\Br(X)$ will be unique up to an element of
$\ker(\Br(X)\to \Br(X_K))$,
the algebraic part of the Brauer group, which has been treated in
\cite{KTeff}.

\

\emph{Step 5.} From the cocycle representatives of elements of
$\Br(X)$ obtained in Step 4, produce sheaves of
Azumaya algebras defined globally on $X$ (Lemma \ref{lem.effAz}).

\

\emph{Step 6.} Compute local invariants.
A sheaf of Azumaya algebras may be effectively converted to a collection
of representing $2$-cocycles, each for a finite \'etale covering of
some $U_i$ with $(U_i)$ a Zariski covering of $X$
(\cite[Thm.\ I.5.1(iii), 5.10]{gb}).
Then we are reduced to the local analysis described in \cite[\S9]{KTeff}.

\section{Generators of $\Br(X_{\bar k})[n]$ by fibrations}
For the first step,
we produce generators of the $n$-torsion in the Brauer group
of $\overline{X}:=X_{\bar k}$.
Starting from $X\subset\PP^N$, a general projection to $\PP^1$
yields, after replacing $X$ by its blow-up at finitely
many points, a fibration
\begin{equation}
\label{fibrwithsection}
f\colon X\to \PP^1
\end{equation}
with geometrically connected fibers.
By removing the exceptional divisors from the codimension 2 complement
in Step 5 and viewing it as a codimension 2 complement of $X$,
the proof of Theorem \ref{thm.main} is reduced to the case that
$f$ as in \eqref{fibrwithsection} exists.

Notice that, given a
finite set of divisors on $X$, \eqref{fibrwithsection} may be chosen so that
each of these divisors maps dominantly to $\PP^1$.

That the $n$-torsion in the Brauer group of a smooth projective surface
over $\bar k$ may be computed using a fibration is standard.
We include a sketch of a proof, for completeness.

\begin{prop}
\label{Brbyfibr}
Let $X$ be a smooth projective geometrically irreducible surface over a number field $k$,
and let $f\colon X\to \PP^1$ be a nonconstant morphism with connected geometric fibers,
both given by explicit equations.
Let $n$ be a given positive integer.
Then there exist, effectively:
\begin{itemize}
\item[(i)] a finite Galois extension $K$ of $k$,
\item[(ii)] a nonempty open subset $S\subset \PP^1$,
\item[(iii)] an \'etale covering $S'\to S$,
\item[(iv)] $2$-cocycles of rational functions for the covering
$X_K\times_{\PP^1_K}S'_K\to X_K\times_{\PP^1_K}S_K$,
\end{itemize}
such that $\Br(X_{\bar k}\times_{\PP^1_{\bar k}}S_{\bar k})[n]$ is spanned by the
classes of the $2$-cocycles, base-extended to $\bar k$.
\end{prop}

\begin{proof}
We let $S\subset \PP^1$ denote the maximal subset over which
$f$ is smooth, and $X^\circ=f^{-1}(S)$.
By the exact sequences of Section \ref{sec.picardschemes},
it suffices to carry out following tasks (perhaps for a larger value of $n$):
\begin{itemize}
\item[(1)] Compute $H^1(S_{\bar{k}},\Pic_{X^\circ_{\bar{k}}/S_{\bar{k}}}[n])$ by means of cocycles.
\item[(2)] Find divisors on $X_{\bar k}$ whose classes in
$\Pic(X^\circ_{\bar k}/S_{\bar k})$ represent the elements appearing in these cocycles.
\item[(3)] Find explicit $2$-cocycle representatives of elements
of $\Br(X^\circ_{\bar k})$ which correspond to these elements by
the isomorphism \eqref{eqn.lerayexaseq}.
\end{itemize}

The field $K$ is an explicit suitable extension, over which
the steps are carried out.
Step (1) is clear, since there is an explicit finite \'etale covering
$C\to S$ trivializing $\Pic_{X^\circ/S}[n]$.
Then there is a finite \'etale covering $S'\to C$, with
$S'_{\bar k}\to C_{\bar k}$ a product of cyclic \'etale degree $n$ covers,
such that $S'_{\bar k}\to S_{\bar k}$ trivializes
$H^1(S_{\bar k},\Pic_{X^\circ_{\bar k}/S_{\bar k}}[n])$.
(The proof of Lemma \ref{lem.H1Cmun} provides an effective procedure
to compute $S'$, using effective Jacobian arithmetic.)
Step (2) can be carried out effectively as described in
\cite[\S 4]{KTeff}, using an effective
version of Tsen's theorem (for the function field, this is
standard, see e.g.\ \cite{preu}, then apply Lemma \ref{lem.effBring}).
On the level of cocycles, the Leray spectral sequence
\eqref{eqn.lerayexaseq} gives rise to a 3-cocycle, and Step (3) can be
carried out as soon as this is represented as a coboundary, which
we have again possibly after making a Zariski refinement of $S'$
(cf.\ \cite[Exa.\ III.2.22(d)]{milne}).
An explicit description of the procedure to produce the 3-cocycle
using the Leray spectral sequence may be found in
\cite[Prop.\ 6.1]{KTeff}.
\end{proof}

\section{Relations among generators}
\label{relations}
In this section we show how to compare
elements of the Brauer group of a Zariski open subset of
a smooth projective surface $\overline{X}$ over $\bar k$,
under the assumption that the geometric Picard group
$\Pic(\overline{X})$ is
finitely generated, and $\overline{X}$ as well as a finite set of
divisors generating $\Pic(\overline{X})$ are explicitly given.
The method goes back to Brauer \cite{Bra28}, with refinements in \cite{Bra32}.

\begin{lemm}
\label{lemm.relation}
Let $X^\circ$ be a smooth quasi-projective geometrically irreducible surface over a number field $k$,
$Z^\circ\to X^\circ$ a finite \'etale morphism,
and $\beta\in Z^2(Z^\circ\to X^\circ, \G_m)$ a \v{C}ech cocycle representative
of an element $\alpha\in \Br(X^\circ)$.
We suppose $X^\circ$, $Z^\circ$ and $\beta$ are given by explicit
equations, respectively functions.
Let $n$ be a given positive integer, and $\gamma\in C^1(Z^\circ\to X^\circ, \G_m)$
a \v{C}ech cochain whose coboundary is equal to $n\cdot \beta$.
We suppose that $k$ contains the $n$-roots of unity, and that an identification
$\mu_n\simeq \Z/n\Z$ is fixed.
Then there exists, effectively, a finite group $G$, a finite \'etale morphism
$Y^\circ\to Z^\circ$,
a $G$-torsor structure on
$Y^\circ\to X^{\circ}$, a central extension of finite groups
\begin{equation}
\label{eq.gpext}
1\to \Z/n\Z\to H\to G\to 1,
\end{equation}
and a $1$-cochain $\delta\in C^1(Y^\circ\to X^{\circ},\G_m)$
such that image in
$\Br(X^{\circ})=H^2(X^{\circ},\G_m)$
of the induced element of $H^2(X^{\circ},\Z/n\Z)\simeq
H^2(X^{\circ},\mu_n)$ is equal to $\alpha$,
and the coboundary $\delta$ is the difference between
the latter and the given $\beta$ refined by
$Y^\circ\to Z^\circ$.
\end{lemm}

\begin{proof}
There exists a finite \'etale cover $Y^\circ$ of $Z^\circ$,
such that the restriction of $\gamma$ to
$Z^\circ\times_{X^\circ}Z^\circ$ is an $n$th power.
It follows that the restriction of $\beta$
differs by a coboundary from an element of the image of
$Z^2(Y^\circ\to X^\circ, \mu_n)$.
These can be produced explicitly.
Upon further refinement of $Y^\circ$, we may suppose that
$Y^\circ$ is irreducible, $Y^\circ\to X^\circ$ is a Galois $G$-covering
for some finite group $G$,
and then the cocycle condition is precisely the condition to be
a 2-cocycle for the group cohomology of $G$ with values in
$\Z/n\Z$ (with trivial $G$-action on $\Z/n\Z$).
This gives us \eqref{eq.gpext}.
\end{proof}

\begin{prop}
\label{prop.relation}
Let $X$ be a smooth projective geometrically irreducible surface over a number field $k$ with finitely generated
geometric Picard group $\Pic(X_{\bar k})$.
Let $X^\circ$ be a nonempty open subscheme,
$\pi\colon Y^\circ\to X^\circ$ an \'etale cover,
and $\beta\in Z^2(Y^\circ\to X^\circ, \G_m)$ a \v{C}ech cocycle representative
of an element $\alpha\in \Br(X^\circ)$.
Let $n$ be a given positive integer, and $\gamma\in C^1(Y^\circ\to X^\circ, \G_m)$
a \v{C}ech cochain whose coboundary is equal to $n\cdot \beta$.
We suppose $X$, a finite set of divisors generating $\Pic(X_{\bar k})$,
$X^\circ$, $Y^\circ$, $\pi$, $\beta$, and $\gamma$ are given by explicit
equations, respectively functions.
Then there exists an effective procedure to determine whether
$\alpha_{\bar k}=0$ in $\Br(X^\circ_{\bar k})$, and in case $\alpha_{\bar k}=0$,
to produce a finite extension $K$ of $k$, a Zariski open covering
$(Y^\circ_i)$ of $Y^\circ$,
and a $1$-cochain
$\delta\in C^1(\coprod (Y^{\circ}_i)_K\to X^\circ_K,\G_m)$,
whose coboundary is equal to the
base-extension to $K$ of the refinement of $\beta$ by
$\coprod Y^\circ_i\to Y^\circ$.
\end{prop}

\begin{proof}
It suffices to prove the result
after an effective shrinking of $X^\circ$ and extension of the
base field, by Lemma \ref{lem.effBring}
(we note that the Zariski open subsets that are produced
in the proof may be taken to be Galois invariant) and, by
Lemma \ref{lem.cobounda}, after a refinement of the given cover.
So we may suppose that $\pi$ is finite, $\Pic(X^\circ_{\bar k})=0$,
the field $k$ contains the $n$th roots of unity
(with a fixed identification $\Z/n\Z\simeq \mu_n$), and
the cocycle $\beta$ takes its values in
$\mu_n$ (Lemma \ref{lemm.relation}) and is the universal one for a
$G$-torsor structure on $Y^\circ\to X^\circ$ and an extension
\eqref{eq.gpext}.
Without loss of generality, $Y^\circ$ is geometrically irreducible,
and the class of the extension in $H^2(G,\Z/n\Z)$
(group cohomology for $\Z/n\Z$ with trivial $G$-action)
is not annihilated by any positive integer smaller than $n$.
It follows from $\Pic(X^\circ_{\bar k})=0$ that
$\alpha_{\bar k}=0$ in $\Br(X^\circ_{\bar k})$ if and only if
the class of $\beta$ is $0$ in $H^2(X^\circ_{\bar k},\mu_n)$.

The Leray spectral sequence gives rise to an exact sequence
$$
0\to H^1(G,\mu_n)\to H^1(X^\circ_{\bar k},\mu_n)\to H^1(Y^\circ_{\bar k},\mu_n)^G\to
H^2(G,\mu_n)\to H^2(X^\circ_{\bar k},\mu_n).
$$
It follows that the class of $\beta$ is $0$ in $\Br(X^\circ_{\bar k})$
if and only if there exists
an irreducible finite \'etale covering $\overline{V}^\circ$ of
$\overline{Y}^\circ:=Y^\circ_{\bar k}$,
cyclic of degree $n$, admitting a structure of
$H$-torsor over $\overline{X}^\circ:=X^\circ_{\bar k}$
compatible with the $G$-torsor structure on $\overline{Y}^\circ$.
This can be tested, provided that we can explicitly generate all
degree $n$ cyclic \'etale coverings of $\overline{Y}^\circ$.
If we have such a covering,
we take $K$ so that the covering and $H$-torsor structure are defined over $K$,
then the restriction of $\beta$ to the covering is explicitly a
coboundary.

Choose an explicit fibration $\tau\colon \overline{Y}^\circ\to\PP^1$.
Since we may shrink $\overline{Y}^\circ$, we may replace $\overline{Y}^\circ$
by the preimage of Zariski open $T\subsetneq\PP^1$, chosen so that the
geometric fibers are complements of exactly some number $\ell$
of distinct points in a smooth irreducible curve of some genus $g$,
these $\ell$ points being the fibers of a finite \'etale cover of $T$.

By the Leray spectral sequence, we have a commutative diagram
with exact rows
$$
\xymatrix@C=15pt{
0 \ar[r] &
{\mathcal{O}(T)^*/(\mathcal{O}(T)^*)^n} \ar[r] \ar[d] &
H^1(\overline{Y}^\circ, \mu_n) \ar[r] \ar[d] &
H^0(T, R^1\tau_*\mu_n) \ar[r] \ar[d] &
0 \\
0 \ar[r] &
{\bar k(T)^*/(\bar k(T)^*)^n} \ar[r] &
H^1(\overline{Y}^\circ_{\bar k(T)},\mu_n) \ar[r] &
H^0(\bar k(T), R^1(\tau_{\bar k(T)})_*\mu_n) \ar[r] &
0
}$$
where we have used $H^2(T,\mu_n)=H^2(\bar k(T),\mu_n)=0$.
Arguing as in \cite[proof of Lemma III.3.15]{milne} we see that
$R^1\tau_*\mu_n$ is a locally constant torsion sheaf with finite fibers.
It follows that the right-hand vertical map is an isomorphism.
By the snake lemma, the leftmost two vertical maps are injective
and have isomorphic cokernels.

Let $\overline{Y}_{\bar k(T)}$ be a smooth projective curve
containing $\overline{Y}^\circ_{\bar k(T)}$ as an open subscheme.
We can find generators of
$H^1(\overline{Y}_{\bar k(T)},\mu_n)$,
modulo $H^1(\bar k(T),\mu_n)=\bar k(T)^*/(\bar k(T)^*)^n$
by computing the $\bar k(T)$-rational $n$-torsion points of the Jacobian
and (using effective Tsen's theorem) lifting these to divisor representatives.
Lemma \ref{lem.H1Cmun} supplies additional generators of
$H^1(\overline{Y}^\circ_{\bar k(T)},\mu_n)$:
for elements of $\bigoplus \Z/n\Z$ (sum over points of
$\overline{Y}_{\bar k(T)}\smallsetminus \overline{Y}^\circ_{\bar k(T)}$)
of weighted (by degree) sum $0$, we test whether a fiber of a multiplication by $n$
map of Jacobians has a $\bar k(T)$-rational point (and again use effective Tsen's
theorem to produce divisor representatives).
Each generator of $H^1(\overline{Y}^\circ_{\bar k(T)},\mu_n)$, modulo
$H^1(\bar k(T),\mu_n)$, may be effectively lifted to
$H^1(\overline{Y}^\circ,\mu_n)$ by a diagram chase, using the
isomorphism of the cokernels of
left two vertical morphisms in the diagram.
\end{proof}

\section{Galois invariants in $\Br(X_{\bar k})[n]$}
\label{sec.galinv}

In this section, we focus on the problem of deciding whether a
Galois invariant element in $\Br(X_{\bar k})[n]$ lies in
the image in $\Br(X_{\bar k})$ of an element of
$\Br(X_K)^{\Gal(K/k)}$.
Concretely, this amounts to adjusting elements of
$\Br(X_K)$ by elements of $\Br(K)$, when possible, so that they
become $\Gal(K/k)$-invariant.
This will be done effectively.

\begin{prop}
\label{prop.extend}
Let $X$ be a smooth projective geometrically irreducible surface over
a number field $k$.
Let $K$ be a finite Galois extension of $k$.
Let $X^\circ$ and $\widetilde{X}^\circ$ be open subschemes whose union is
the complement of a subset that has codimension $2$ (or is empty),
$Y^\circ\to X^\circ$ and
$\widetilde{Y}^\circ\to\widetilde{X}^\circ$ \'etale coverings,
$\beta\in Z^2(Y^\circ_K\to X^\circ_K,\G_m)$ and
$\tilde\beta\in Z^2(\widetilde{Y}^\circ_K\to\widetilde{X}^\circ_K,\G_m)$
cocycles,
and
$\delta_g\in C^1(Y^\circ_K\to X^\circ_K,\G_m)$
having coboundary $\beta-{}^g\beta$ for every
$g\in\Gal(K/k)$.
Assume that $\beta$ and $\tilde\beta$ give rise to the same
class in $\Br((X^\circ\cap \widetilde{X}^\circ)_{\bar{k}})$.
Then we may effectively produce an open subscheme
$\widehat{X}\subset X$, containing $X^\circ$,
whose complement has codimension $2$ (or is empty),
an \'etale cover $\widehat{Y}\to \widehat{X}$,
a finite extension $L$ of $K$, Galois over $k$,
a cocycle $\hat\beta\in Z^2(\widehat{Y}\to \widehat{X},\G_m)$
giving rise to the same class as $\beta$
in $\Br(X^\circ_{\bar{k}})$,
and cochain $\hat\delta_g\in C^1(\widehat{Y}\to \widehat{X},\G_m)$
having coboundary $\hat\beta-{}^g\hat\beta$,
for all $g\in\Gal(L/k)$.
\end{prop}

\begin{proof}
Let $\xi_1$, $\ldots$, $\xi_N$ denote the codimension 1
generic points of
$Y^\circ\times_X\widetilde{Y}^\circ$ whose image in
$\tilde{X}^\circ$ is one of the generic points in $X$
of the codimension 1 irreducible components of
$X\smallsetminus X^\circ$.
We may apply Proposition \ref{prop.relation} to the covering
$Y^\circ\times_X\widetilde{Y}^\circ\to X^\circ\cap \widetilde{X}^\circ$
and insist that one of the open sets that is produced, in
addition to being Galois invariant, contains all the points
above $\xi_1$, $\ldots$, $\xi_N$.
(The field that emerges, enlarged if necessary, is taken as the
field $L$ mentioned in the statement.)
Call the open set $U$.
We replace $\widetilde{X}^\circ$ with the complement of the closure
of the image of the complement of $U$ in
$Y^\circ_L\times_{X_L}\widetilde{Y}^\circ_L$,
and restrict $\widetilde{Y}^\circ$ accordingly.
Now we have $U=Y^\circ_L\times_{X_L}\widetilde{Y}^\circ_L$,
so we may apply Lemma \ref{lem.effmv} to produce
$\hat\beta\in Z^2(Y^\circ_L\amalg \widetilde{Y}^\circ_L\to
X^\circ_L\cup \widetilde{X}^\circ_L,\G_m)$.
We apply Lemma \ref{lem.effBring} to produce
$\hat\delta_g$, which involves replacing
$Y^\circ$ and $\widetilde{Y}^\circ$ by Zariski covers.
\end{proof}

\begin{prop}
\label{prop.descent}
Let $X$ be a smooth projective geometrically irreducible
variety over a number field $k$, given by explicit equations,
let $K$ be a finite Galois extension of $k$,
and assume that $\Pic(X_{\bar k})$ is
torsion-free, generated by finitely many explicitly given divisors,
defined over $K$.
Let $\alpha\in\Br(X_{\bar k})$ be given by means of a
cocycle representative
$\beta\in Z^2(\widehat{Y}_K\to \widehat{X}_K,\G_m)$,
where
$\widehat{Y}_K\to \widehat{X}_K$ is an \'etale cover, with
$\widehat{X}$ an open subscheme of $X$ whose complement has
codimension at least $2$ (or is empty).
Assume given
$\delta^{(g)}\in C^1(\widehat{Y}_K\to \widehat{X}_K,\G_m)$, having
coboundary $\beta-{}^g\beta$, for every $g\in \Gal(K/k)$.
Then there exists an effective computable obstruction
in $H^2(\Gal(K/k),\Pic(X_K))$ to the existence of
$\alpha_0\in\Br(X)$ such that
$\alpha_0$ and $\alpha$ have the same image in $\Br(X_{\bar k})$.
When the obstruction class vanishes, we can effectively
construct a cocycle representative of $\alpha_0|_{\widehat{X}}$
in $Z^2(\widehat{Y}_K\to \widehat{X},\G_m)$ for some
$\alpha_0\in\Br(X)$ satisfying
$(\alpha_0)_K=\alpha$.
\end{prop}

\begin{proof}
By the Leray spectral sequence, we have an exact sequence
$$\Br(X)\to \ker\left(\Br(X_K)^{\mathrm{\Gal(K/k)}}\to H^2(\Gal(K/k),\Pic(X_K))\right)
\to H^3(\Gal(K/k),K^*).$$
Also note that the nontriviality in $H^2(\Gal(K/k),\Pic(X_K))$
implies the nontriviality in $H^2(\Gal(L/k),\Pic(X_L))$ for any
finite extension $L$ of $K$, Galois over $k$,
by the Hochschild-Serre spectral sequence
\begin{align*}
0=H^1(\Gal(L/K),\Pic(X_L))^{\Gal(K/k)}&\to \\
H^2(\Gal(K/k),&\Pic(X_K))\to
H^2(\Gal(L/k),\Pic(X_L)).
\end{align*}

The hypothesis concerning $\delta^{(g)}$ may be written
\begin{equation}
\label{eq.b1}
\frac{\delta^{(g)}(y,y')\delta^{(g)}(y',y'')}{\delta^{(g)}(y,y'')}=
\frac{\beta(y,y',y'')}{{}^g\beta(y,y',y'')}
\end{equation}
and implies that
\begin{equation}
\label{eq.c1}
\frac{\delta^{(g)}\,{}^g\delta^{(g')}}{\delta^{(gg')}}
\in Z^1(\widehat{Y}_K\to \widehat{X}_K,\G_m)
\end{equation}
for every $g$, $g'\in\Gal(K/k)$.
Arguments as in \cite[\S 6]{KTeff} show that
\eqref{eq.c1} gives the obstruction class in
$H^2(\Gal(K/k),\Pic(X_K))$.
Of course, each cocycle \eqref{eq.c1} may be
explicitly represented by a divisor, whose class in $\Pic(X_K)$ is then
readily computed.

Assuming that the obstruction class in
$H^2(\Gal(K/k),\Pic(X_K))$ vanishes,
each $\delta^{(g)}$ may be modified by a cocycle so that
each element \eqref{eq.c1} is a coboundary, i.e., so that
there exist
$\varepsilon^{(g,g')}\in \mathcal{O}^*_{\widehat{Y}_K}$ satisfying
\begin{equation}
\label{eq.b2}
\frac{\varepsilon^{(g,g')}(y')}{\varepsilon^{(g,g')}(y)}=
\frac{\delta^{(g)}(y,y'){}^g\delta^{(g')}(y,y')}{\delta^{(gg')}(y,y')}.
\end{equation}
In this case the divisor representative of \eqref{eq.c1} is
a principal divisor, hence the divisor associated to an
effectively computable rational function.

Combining \eqref{eq.b1} and \eqref{eq.b2}, we have
$$\frac{\varepsilon^{(g,g')}(y)\varepsilon^{(gg',g'')}(y)}{\varepsilon^{(g,g'g'')}(y){}^g\varepsilon^{(g',g'')}(y)}=
\frac{\varepsilon^{(g,g')}(y')\varepsilon^{(gg',g'')}(y')}{\varepsilon^{(g,g'g'')}(y'){}^g\varepsilon^{(g',g'')}(y')},$$
hence
\begin{equation}
\label{eq.c2}
\varepsilon^{(g,g')}\varepsilon^{(gg',g'')}/(\varepsilon^{(g,g'g'')}\,{}^g\varepsilon^{(g',g'')})\in \mathcal{O}^*_{\widehat{X}_K},
\end{equation}
i.e., is a constant function,
for every $g$, $g'$, $g''\in\Gal(K/k)$.
The rest of the argument is similar to \cite[Prop.\ 6.3]{KTeff}. The
constants \eqref{eq.c2} determine a class
in $H^3(\Gal(K/k),K^*)$, which may be effectively tested for vanishing.
In case of nonvanishing a further finite extension may be effectively
computed, which kills this class.
In case of vanishing, a $2$-cochain lift is effectively produced.
Modifying $\varepsilon^{(g,g')}$, then, yields
\begin{equation}
\label{eq.b3}
\varepsilon^{(g,g')}(y)\varepsilon^{(gg',g'')}(y)=
\varepsilon^{(g,g'g'')}(y){}^g\varepsilon^{(g',g'')}(y)
\end{equation}
Now if we set
$$\beta^{(g,g')}(y,y',y'')=
\frac{\beta(y,',y'')\varepsilon^{(g,g')}(y'')}{\delta^{(g)}(y',y'')}$$
then we have
$$\beta^{(g,g')}(y,y',y'')\beta^{(gg',g'')}(y,y'',y''')=
\beta^{(g,g'g'')}(y,y',y'''){}^g\beta^{(g',g'')}(y',y'',y'''),$$
i.e., we have an element of $Z^2(\widehat{Y}_K\to \widehat{X},\G_m)$
determining an element $\alpha_0\in H^2(X,\G_m)$.
The restriction to $\widehat{X}_K$ is defined by the cocycle
$\beta^{(e,e)}$, which is equal to $\beta$, up to coboundary.
\end{proof}

\end{document}